\documentclass[leqno]{amsart}
\usepackage{amsmath,amsthm,amssymb}
\usepackage[T1]{fontenc}

\newtheorem{Th}{Theorem}
\newtheorem{Prop}[Th]{Proposition}
\newtheorem{Lemma}[Th]{Lemma}
\newtheorem{Cor}[Th]{Corollary}

\theoremstyle{remark}
\newtheorem{Remark}{Remark}
\newtheorem*{Example}{Example}
\newtheorem*{Def}{Definition}
\newtheorem{Prob}{Problem}

\newcommand{\cb}{\mathcal{ B}}
\newcommand{\ca}{\mathcal{ A}}
\newcommand{\cc}{\mathcal{ C}}

\newcommand{\cm}{\mathcal{ M}}

\newcommand{\R}{\mathbb{{R}}}
\newcommand{\T}{\mathbb{{T}}}

\newcommand{\Z}{\mathbb{{Z}}}

\newcommand{\xbm}{(X,\mathcal{ B},\mu)}

\newcommand{\ov}{\overline}

\newcommand{\va}{\varphi}

\newcommand{\ot}{\otimes}
\newcommand{\la}{\lambda}

\newcommand{\bez}{\nopagebreak\hspace*{\fill}\nolinebreak$\Box$}

\begin{document}

\title{A note on quasi-similarity of Koopman operators}
\author{K. Fr\k{a}czek
\and M. Lema\'nczyk}
\address{K.\ Fr\k{a}czek \and M.\ Lema\'nczyk\\ Faculty of Mathematics and Computer Science\\ Nicolaus
Copernicus University\\ ul. Chopina 12/18, 87-100 Toru\'n, Poland}

\email{fraczek@mat.uni.torun.pl, mlem@mat.uni.torun.pl}

\subjclass[2000]{37A05, 37A30, 37A35}

\thanks{Research partially supported by Polish MNiSzW grant
N N201 384834; partially supported by Marie Curie ``Transfer of
Knowledge'' EU program -- project MTKD-CT-2005-030042 (TODEQ)}

\begin{abstract} Answering a question of A. Vershik we construct two non-weakly isomorphic ergodic
automorphisms for which the associated unitary (Koopman)
representations are Markov quasi-similar. We also discuss metric
invariants of Markov quasi-similarity in the class of ergodic
automorphisms.
\end{abstract}

\maketitle

\section*{Introduction}
Markov operators appear in the classical ergodic theory in the
context of joinings, see the monograph \cite{Gl}. Indeed, assume
that $T_i$ is an ergodic automorphism of a standard probability
Borel space $(X_i,\mathcal{B}_i,\mu_i)$, $i=1,2$. Consider $\la$ a
{\em joining} of $T_1$ and $T_2$, i.e.\ a $T_1\times
T_2$-invariant probability measure on $(X_1\times
X_2,\cb_1\ot\cb_2)$ with the marginals $\mu_1$ and $\mu_2$
respectively. Then the operator
$\Phi_{\la}:L^2(X_1,\cb_1,\mu_1)\to L^2(X_2,\cb_2,\mu_2)$
determined by \begin{equation}\label{q1} \langle \Phi_\lambda
f_1,f_2\rangle_{L^2(X_2,\cb_2,\mu_2)}=\langle f_1\ot{\mathbf
1}_{X_2},{\mathbf 1}_{X_1}\ot f_2\rangle_{L^2(X_1\times X_2
,\cb_1\ot\cb_2,\lambda)}\end{equation} is Markov (i.e.\ it is a
linear contraction which preserves the cone of non-negative
functions and
$\Phi_{\la}{\mathbf1}={\mathbf1}=\Phi_{\la}^\ast{\mathbf 1}$) and
moreover \begin{equation}\label{q2} \Phi_\la\circ
U_{T_1}=U_{T_2}\circ\Phi_\la,\end{equation} where
$U_{T_i}:L^2(X_i,\cb_i,\mu_i)\to L^2(X_i,\cb_i,\mu_i)$  stands for
the associated unitary operator: $U_{T_i}f=f\circ T_i$ for $f\in
L^2(X_i,\cb_i,\mu_i)$, $i=1,2$, which is often called a {\em
Koopman operator}. In fact, each Markov operator
$\Phi:L^2(X_1,\cb_1,\mu_1)\to L^2(X_2,\cb_2,\mu_2)$ satisfying the
equivariance property~(\ref{q2}) is of the form $\Phi_{\la}$ for a
unique joining $\la$ of $T_1$ and $T_2$ (see e.g.\
\cite{Le-Pa-Th}, \cite{Ry}). Markov operators corresponding to
ergodic joinings are called {\em indecomposable}.

In order to classify dynamical systems one usually considers the
{\em measure-theoretic isomorphism}, i.e.\ the equivalence given
by the existence of an invertible map
$S:(X_1,\cb_1,\mu_1)\to(X_2,\cb_2,\mu_2)$ for which $S\circ
T_1=T_2\circ S$. The measure-theoretic (metric) isomorphism
implies spectral equivalence of the Koopman operators $U_{T_1}$
and $U_{T_2}$; indeed, $U_{S^{-1}}$ (where $U_{S^{-1}}f_1=f_1\circ
S^{-1}$ for $f_1\in L^2(X_1,\cb_1,\mu_1)$) provides such an
equivalence. The converse does not hold, see e.g.\ \cite{An}; we
also recall that all Bernoulli shifts are spectrally equivalent
while the entropy classify them measure-theoretically \cite{Or}.
One may ask whether there can be some other natural classification
of dynamical systems which lies in between metric and spectral
equivalence.

In \cite{Ve}, A. Vershik considers the quasi-similarity problem in
the class of Markov operators. Recall that if $A_i$ is a bounded
linear operator of a Hilbert space $H_i$, $i=1,2$, and if there is
a bounded linear operator $V:H_1\to H_2$ whose range is dense and
which intertwines $A_1$ and $A_2$, then $A_2$ is said to be a {\em
quasi-image} of $A_1$. By duality, $A_2$ is a quasi-image of $A_1$
if and only if there exists a $1-1$ bounded linear  operator
$W:H_2\to H_1$ intertwining $A_2$ and $A_1$. If also $A_1$ is a
quasi-image of $A_2$ then the two operators are called {\em
quasi-similar}. The main problem taken up in \cite{Ve} is a study
of quasi-similarity in the class of Markov operators, i.e.\ given
a Markov operator $A_i$ of $H_i=L^2(X_i,\cb_i,\mu_i)$, $i=1,2$, we
investigate a possible quasi-similarity of $A_1$ and $A_2$, where
we additionally require $V$ to be a Markov operator between the
corresponding $L^2$-spaces. In what follows, we will call such a
quasi-similarity {\em Markov quasi-similarity}.

Notice that each  Koopman operator is also a Markov operator. It
is known (see e.g.\ \cite{Le-Pa}, \cite{Ve}) that if an
intertwining Markov operator $\Phi:L^2(X_1,\cb_1,\mu_1)\to
L^2(X_2,\cb_2,\mu_2)$  is unitary then it has to be of the form
$U_S$ where $S$ provides a measure-theoretic isomorphism. On the
other hand the quasi-similarity of unitary operators implies their
spectral equivalence (see Section~\ref{spectral} below).
Therefore, Markov quasi-similarity lies in between the spectral
and measure-theoretic equivalence of dynamical systems. One of
questions raised by Vershik in \cite{Ve} is the following:
\begin{equation}\label{qvershik}\begin{array}{l}
\mbox{\em Do there exist two automorphisms that are not
isomorphic}\\\mbox{\em but are Markov
quasi-similar?}\end{array}\end{equation} In order to answer this
question notice that any weakly isomorphic automorphisms (see
\cite{Si}) $T_1$ and $T_2$ are automatically Markov quasi-similar;
indeed, the weak isomorphism means that there are $\pi_1$ and
$\pi_2$ which are homomorphisms between $T_1$ and $T_2$ and $T_2$
and $T_1$ respectively, then $U^\ast_{\pi_1}$ and $U^\ast_{\pi_2}$
yield Markov quasi-similarity of $T_1$ and $T_2$. Hence, if $T_1$
and $T_2$ are weakly isomorphic but not isomorphic, we obtain the
positive answer to the question~(\ref{qvershik}). Examples of
weakly isomorphic but not isomorphic systems are known in the
literature, see e.g.\ \cite{Kw-Le-Ru}, \cite{Le}, \cite{Ru},
including the case of K-automorphisms \cite{Ho}. It follows that
the notion of Markov quasi-similarity has to be considered as an
interesting refinement of the notion of weak isomorphism, and in
Vershik's question~(\ref{qvershik}) we have to replace ``not
isomorphic'' by ``not weakly isomorphic''.

The main aim of this note is to answer positively this modified
question~(\ref{qvershik}) (see Proposition~\ref{vershik-answer}
below). We would like to emphasize that despite a spectral flavor
of the definition, Markov quasi-similarity is far from being the
same as spectral equivalence. For example, partly answering
Vershik's question raised at a seminar at Penn State University in
2004 whether entropy is an invariant of Markov quasi-similarity,
we show that zero entropy as well as K-property  are invariants of
Markov quasi-similarity of automorphisms, while they are not
invariants of spectral equivalence of the corresponding unitary
operators. These facts and related problems will be discussed in
Sections~\ref{AA}-\ref{BB}.

\section{Quasi-similarity of unitary operators implies their
unitary equivalence}\label{spectral} Assume that $U$ is a unitary
operator of a separable Hilbert space $H$. Given $x\in H$ by
$\Z(x)$ we denote the {\em cyclic space generated by} $x$, i.e.\
$\Z(x)=\ov{\mbox{span}}\{U^nx:\:n\in\Z\}$. We will use a similar
notation $\Z(y_1,\ldots,y_k)$ for the smallest closed
$U$-invariant subspace containing $y_i$, $i=1,\ldots,k$. Denote by
$\T$ the (additive) circle. Then the Fourier transform of the
(positive) measure $\sigma_x$ -- called the {\em spectral measure
of} $x$ -- is given by
$$\widehat{\sigma}_x(n):=\int_{\T}e^{2\pi int}\,d\sigma_x(t)=\langle
U^nx,x\rangle\text{ for each }n\in\Z.$$ Similarly the sequence
$(\langle U^nx,y\rangle)_{n\in\Z}$ is the Fourier transform of the
(complex) spectral measure $\sigma_{x,y}$ of $x$ and $y$. Given a
spectral measure $\sigma$ we denote
$$
H_\sigma=\{x\in H:\:\sigma_x\ll\sigma\}.$$ Then $H_\sigma$ is a
closed $U$-invariant subspace called a {\em spectral subspace} of
$H$.

It follows from Spectral Theorem for unitary operators (see e.g.\
\cite{Ka-Th} or \cite{Pa}) that there is a  decomposition
\begin{equation}\label{decomp} H=H_{\sigma_1}\oplus H_{\sigma_2}\oplus\ldots\end{equation}
into spectral subspaces such that for each $i\geq1$ $$
H_{\sigma_i}= \bigoplus_{k=1}^{n_i}\Z(x^{(i)}_k),$$ where
$\sigma_i\equiv\sigma_{x^{(i)}_1}\equiv\sigma_{x^{(i)}_2}\equiv\ldots$
($n_i$ can be infinity), and $\sigma_i\perp\sigma_j$ for $i\neq
j$.
 The class $\sigma_U$ of all finite measures equivalent to the sum
$\sum_{i\geq1}\sigma_i$ is then called the {\em maximal spectral
type of} $U$. Another important invariant of $U$ is the spectral
multiplicity function $M_U:\T\to\{1,2,\ldots\}\cup\{\infty\}$ (see
\cite{Ka-Th}, \cite{Pa}) which is defined $\sigma$-a.e., where
$\sigma$ is any measure belonging to the maximal spectral type of
$U$. Note that decomposition~(\ref{decomp}) is far from being
unique but if $$H=\bigoplus_{i=1}^\infty H_{\sigma'_i},
\;\;H_{\sigma'_i}=\bigoplus_{k=1}^{n'_i}\Z(y_k^{(i)})$$ is another
decomposition~(\ref{decomp}) in which $\sigma_i\equiv\sigma'_i$,
$i\geq1$, then $n_i=n'_i$ for $i\geq1$. Recall that the essential
supremum $m_U$ of $M_U$ (called the {\em maximal spectral
multiplicity of} $U$) is equal to
\begin{equation}\label{qqq} \inf\{m\geq1:\:\Z(y_1,\ldots,y_m)=H\;\;\;\mbox{for
some}\;\;y_1,\ldots,y_m\in H \};\end{equation} if there is no
``good'' $m$, them $m_U=\infty$.

Assume that $U_i$ is a unitary operator of a separable Hilbert
space $H_i$, $i=1,2$. Let $V:H_1\to H_2$ be a bounded linear
operator which intertwines $U_1$ and $U_2$. Then for each $n\in\Z$
and $x_1\in H_1$
$$\langle U_2^nVx_1,Vx_1\rangle=\langle U_1^nx_1,V^\ast
Vx_1\rangle,$$ so by elementary properties of spectral measures
\begin{equation}\label{q3} \sigma_{Vx_1}=\sigma_{x_1,V^\ast
Vx_1}\ll \sigma_{x_1}.\end{equation} Assuming additionally that
Im$(V)$ is dense, an immediate consequence of~(\ref{q3}) is that
the maximal spectral type of a quasi-image of $U_1$ is absolutely
continuous with respect to $\sigma_{U_1}$. It is also clear that
given $y^{(1)}_1,\ldots,y^{(1)}_m\in H_1$ we have
$$\ov{V(\Z(y^{(1)}_1,\ldots,y^{(1)}_m))}=\Z(Vy^{(1)}_1,\ldots,Vy^{(1)}_m).$$
This in turn implies that  the  maximal spectral multiplicity of a
quasi-image of $U_1$ is at most $m_{U_1}$.

\begin{Prop}\label{equivalence} If $U_1$ and $U_2$ are
quasi-similar then they are spectrally equivalent.\end{Prop}
\begin{proof} Assume that $V:H_1\to H_2$ and $W:H_2\to H_1$ intertwine
 $U_1$ and
$U_2$ and have dense ranges. In view of~(\ref{q3}) both operators
$U_1$ and $U_2$ have the same maximal spectral types. Consider a
decomposition~(\ref{decomp}) for $U_1$: $H_1=\bigoplus_{i\geq1}
H_{\sigma^{(1)}_i}$ and let $F_i:=\ov{V(H_{\sigma^{(1)}_i})}$ for
$i\geq1$. The subspaces $F_i$ are obviously $U_2$-invariant and
let $\sigma^{(2)}_i$ ($n^{(2)}_i$) denote the maximal spectral
type (the maximal spectral multiplicity) of $U_2$ on $F_i$. It
follows from~(\ref{q3}) that $\sigma^{(2)}_i\ll\sigma^{(1)}_i$ for
$i\geq1$ and $\sigma^{(2)}_i, \sigma^{(2)}_j$ are mutually
singular (in particular, $F_i\perp F_j$) whenever $i\neq j$.
Moreover, $n^{(2)}_i\leq n^{(1)}_i$, $i\geq1$. Since $V$ has dense
range, $H_2=\bigoplus_{i\geq1} F_i$. It follows that (up to
equivalence of measures) $\sum_{i\geq1}\sigma^{(i)}_2$ is the
maximal spectral type of $U_2$ hence it is equivalent to
$\sum_{i\geq1}\sigma^{(1)}_i$ and therefore
$\sigma^{(1)}_i\equiv\sigma^{(2)}_i$ for $i\geq1$. The same
reasoning applied to the decomposition $H_2=\bigoplus_{i\geq1}F_i$
and $W$ shows that $H_1=\bigoplus_{i\geq1}\ov{W(F_i)}$ and the
maximal spectral type of $U_1$ on $\ov{W(F_i)}$ is absolutely
continuous with respect to $\sigma^{(2)}_i\equiv\sigma_i^{(1)}$,
$i\geq1$. It follows that $\ov{W(F_i)}=H_{\sigma_i^{(1)}}$ for all
$i\geq1$. In particular, we have proved that
$n^{(2)}_i=n^{(1)}_i$ but we need to show that on $F_i$ the
multiplicity is uniform. Suppose this is not the case, i.e. that
for some measure $\eta\ll\sigma^{(2)}_i$ we have
$$
F_i=\Z(z_1)\oplus\ldots\oplus \Z(z_r)\oplus F'_i,$$ where for
$j=1,\ldots r$, $\sigma_{z_j}=\eta$, $1\leq r<n_i^{(2)}$ and the
maximal spectral type of $U_2$ on $F_i'$ is orthogonal to $\eta$.
We have
$$
H_{\sigma_i^{(1)}}=\ov{W(F_i)}=G_i\oplus\ov{W(F'_i)},$$ where
$G_i=\ov{W(\Z(z_1)\oplus\ldots\oplus \Z(z_r))}$ and the maximal
spectral types on $G_i$, say $\tau(\ll\eta)$, and $\ov{W(F'_i)}$
are mutually singular. It follows that the multiplicity of $\tau$
is at most $r$, which is a contradiction since all measures
absolutely continuous with respect to $\sigma^{(1)}_i$ have
multiplicity $n^{(1)}_i$.
\end{proof}

\begin{Remark}\label{similarity}
Literally speaking, the notion of quasi-similarity is  weaker than
the classical notion of quasi-affinity \cite{Fo-Na}: $A_1$ and
$A_2$ are {\em quasi-affine} if there exists a $1-1$ bounded
linear operator $V:H_1\to H_2$ with dense range intertwining $A_1$
and $A_2$. Proposition~3.4 in \cite{Fo-Na} tells us that
quasi-affine unitary operators are unitarily equivalent. Hence
Proposition~\ref{equivalence} shows in fact that for unitary
operators quasi-similarity and  quasi-affinity are equivalent
notions.

Similarly to Markov quasi-similarity of Koopman operators we can
speak about their {\em Markov quasi-affinity}. It is not clear
(see Section~\ref{BB}) whether these two notions coincide.
\end{Remark}

\section{A convolution operator in
\protect$l^2(\Z)$}\label{calkowanie} In this section we produce a
sequence in $l^2(\Z)$ which will be used to construct a Markov
quasi-affinity between two non-weakly isomorphic automorphisms in
Section~\ref{przyklad}.

Denote by $l_0(\Z)$ the subspace of $l^2(\Z)$ of complex sequences
$\bar{x}=(x_n)_{n\in\Z}$ such that $\{n\in\Z:\:x_n\neq 0\}$ is
finite.

\begin{Prop}\label{istciagu}
There exists a nonnegative sequence $\bar{a}=(a_n)_{n\in\Z}\in l^2(\Z)$
such that $\sum_{n\in\Z}a_n=1$ and
\begin{equation}\label{zalpod}
\mbox{ for every }\bar{x}=(x_n)_{n\in\Z}\in l^2(\Z)\mbox{ if
}\bar{a}\ast\bar{x}\in l_0(\Z)\mbox{ then }\bar{x}=\bar{0}.
\end{equation}\end{Prop}

Each element $\ov{y}\in l^2(\Z)$ is an $L^2$-function on $\Z$ and
its Fourier transform is a function $h\in L^2(\T)$ for which
$\widehat{h}(n)=y_n$ for all $n\in\Z$. Moreover, the convolution
of $l^2$-sequences corresponds to the pointwise multiplication of
$L^2$-functions on the circle. It follows that in order to find
the required sequence $\bar{a}$, it suffices to find a function
$f\in L^2(\T)$ such that
\begin{itemize}
\item $a_n=\hat{f}(n)\geq 0$, $\sum_{n\in\Z}a_n=1$;
\item for every $g\in L^2(\T)$, if $f\cdot g=0$ then
$g=0$; \item  for every non-zero trigonometric polynomial $P$, if
$P=f\cdot g$ then  $g\notin L^2(\T)$. \end{itemize} This is done
below.

\begin{Lemma}\label{dodwspol}
If $f:[0,1]\to\R_+$ is a convex $C^2$-function such that
$f(1-x)=f(x)$ for all $x\in[0,1]$ then $\hat{f}(n)\geq 0$ for all
$n\in\Z$.
\end{Lemma}
\begin{proof}
By assumption, $f''(x)\geq 0$ for all $x\in[0,1]$. Using
integration by parts twice, for $n\neq 0$ we obtain
\begin{eqnarray*}
\hat{f}(n)&=&\int_0^1f(x)e^{-2\pi inx}dx=\int_0^1f(x)\cos(2\pi
nx)\,dx=\frac{1}{2\pi n}\int_0^1f(x)\,d\sin(2\pi
nx)\\&=&-\frac{1}{2\pi n}\int_0^1f'(x)\sin(2\pi
nx)\,dx=\frac{1}{4\pi^2 n^2}\int_0^1f'(x)\,d\cos(2\pi
nx)\\
&=&\frac{1}{4\pi^2 n^2}\left[f'(1)-f'(0)-\int_0^1f''(x)\cos(2\pi
nx)\,dx\right]\\
&\geq&\frac{1}{4\pi^2
n^2}\left[f'(1)-f'(0)-\int_0^1|f''(x)\cos(2\pi
nx)|\,dx\right]\\&\geq&\frac{1}{4\pi^2
n^2}\left[f'(1)-f'(0)-\int_0^1f''(x)\,dx\right]=0.
\end{eqnarray*}
\end{proof}

\begin{proof}[Proof of Proposition~\ref{istciagu}]Let us consider $f:[0,1]\to\R$ defined by
\[f(x)=\left\{\begin{array}{ccc}e^{-\frac{1}{|x-1/2|}+2}&\mbox{ if }&x\neq 1/2\\
0&\mbox{ if }&x=1/2.
\end{array}\right.\]
Since $f''(x)\geq 0$ for $x\in[0,1]$, by Lemma~\ref{dodwspol},
$a_n=\hat{f}(n)\geq 0$. As $f:\T\to\R$ is a continuous function of
bounded variation,
\[1=f(0)=\sum_{n\in\Z}a_n.\]

Since $f(x)\neq 0$ for $x\neq12$, if $f\cdot g=0$ for some $g\in
L^2(\T)$ then $g=0$.

Suppose, contrary to our claim, that there exist $g\in L^2(\T)$
and a non-zero trigonometric polynomial $P$ such that $f\cdot
g=P$. Recall that for every $m\geq 0$ we have $\int_0^1
e^{1/x}x^m\,dx=+\infty$, hence $\int_0^1
(e^{1/x}x^m)^2\,dx=+\infty$. Since $P$ is a non-zero analytic
function, there exists $m\geq 0$ such that $P^{(m)}(1/2)\neq 0$
and $P^{(k)}(1/2)= 0$ for $0\leq k<m$. By Taylor's formula, there
exist $C>0$ and $0<\delta<1/2$ such that $|P(x+1/2)|\geq C|x|^m$
for $x\in[-\delta,\delta]$. It follows that
\begin{eqnarray*}\int_\T|g(x)|^2\,dx&\geq &
\int_{1/2}^{1/2+\delta}|P(x)|^2/f(x)^2\,dx=\int_0^{\delta}|P(x+1/2)|^2/f(x+1/2)^2\,dx\\
&\geq&\int_0^\delta(Cx^me^{1/x})^2\,dx=+\infty, \end{eqnarray*}
and hence $g\notin L^2(\T)$ which completes the proof.
\end{proof}

\section{Two non-weakly isomorphic automorphisms which are Markov
quasi-similar}\label{przyklad}

Let $T$ be an ergodic automorphism of $\xbm$. Assume that $G$ is a
compact metric Abelian group with Haar measure $\la_G$. A
measurable function $\va:X\to G$ is called a {\em cocycle}. Using
the cocycle we can define a {\em group extension} $T_{\va}$ of $T$
which acts on $(X\times G,\cb\ot\cb(G),\mu\ot\la_G)$ by the
formula $T_{\va}(x,g)=(Tx,\va(x)+g)$.

We will first take $\va:X\to\Z_2:=\{0,1\}$ so that the group
extension $T_{\va}$ is ergodic. Then assume that we can find $S$
acting on $\xbm$, $ST=TS$, such that if we put $G=\Z_2^{\Z}$ and
define
$$
\psi:X\to G
,\;\;\psi(x)=(\ldots,\varphi(S^{-1}x),\stackrel{0}{\varphi(x)},\varphi(Sx),\varphi(S^2x),\ldots)$$
then $T_\psi$ is ergodic as well (see \cite{Le} for concrete
examples of $T$, $\va$ and $S$ fulfilling our requirements). Put
now $T_1=T_\psi$ and let us take a factor $T_2$ of $T_1$ obtained
by ``forgeting'' the first $\Z_2$-coordinate. In other words on
$(X\times\Z_2^\Z,\mu\ot\la_{\Z_2^{\Z}})$ we consider two
automorphisms
\[T_1(x,\underline{i})=(Tx,\ldots,i_{-1}+\varphi(S^{-1}x),\stackrel{0}
{i_0+\varphi(x)},i_1+\varphi(Sx),i_2+\varphi(S^2x),\ldots),\]
\[T_2(x,\underline{i})=(Tx,\ldots,i_{-1}+\varphi(S^{-1}x),\stackrel{0}
{i_0+\varphi(x)},i_1+\varphi(S^2x),i_2+\varphi(S^3x),\ldots),\]
where
$\underline{i}=(\ldots,i_{-1},\stackrel{0}{i_0},i_1,i_2,\ldots)$.
Define $I_n:X\times\Z_2^\Z\to X\times\Z_2^\Z$ by putting
\[I_n(x,\underline{i})=(S^nx,\ldots,i_{n-1},\stackrel{0}{i_n},i_{n+2},i_{n+3},\ldots).\]
Then $I_n$ is measure-preserving and $I_n\circ T_1=T_2\circ I_n$.
Therefore \begin{equation}\label{aaaa}U_{T_1}\circ
U_{I_n}=U_{I_n}\circ U_{T_2}\end{equation} with $U_{I_n}$ being an
isometry (which is not onto) and
\begin{eqnarray*}\lefteqn{U^*_{I_n}F(x,\underline{i})}\\&=&
\frac{1}{2}\left(F(S^{-n}x,\ldots,\stackrel{0}{i_{-n}},\ldots,\stackrel{n}{i_0},0,i_{1},\ldots)
+F(S^{-n}x,\ldots,\stackrel{0}{i_{-n}},\ldots,\stackrel{n}{i_0},1,i_{1},\ldots)\right).
\end{eqnarray*}

Let $\bar{a}=(a_n)_{n\in\Z}\in l^2(\Z)$ be a nonnegative sequence
such that $\sum_{n\in\Z}a_n=1$ and~(\ref{zalpod}) holds. Let
$J:L^2(X\times\Z_2^\Z,\mu\ot\la_{\Z_2^{\Z}})\to
L^2(X\times\Z_2^\Z,\mu\ot\la_{\Z_2^{\Z}})$ stand for the Markov
operator defined by
\[J=\sum_{n\in\Z}a_nU_{I_n}.\]
In view of~(\ref{aaaa}), $J$ intertwines $U_{T_1}$ and $U_{T_2}$.

Denote by $Fin$ the set of finite nonempty subsets of $\Z$. Let us
consider two operations on $Fin$:
\[\widehat{A}=\{s\in A:\:s\leq 0\}\cup\{s+1:\:s\in A,s>0\}\text{ for }A\in Fin;\]
\[\widetilde{B}=\{s\in B:\:s\leq 0\}\cup\{s-1:\:s\in B,s>1\}\text{ for }B\in Fin\text{ with }1\notin B.\]
Of course, $\widetilde{\widehat{A}}=A$ and
$\widehat{\widetilde{B}}=B$. Let  $\sim$ stand for the equivalence
relation in $Fin$ defined by $A\sim B$ if $A=B+n$ for some
$n\in\Z$. Denote by $Fin_0$ a fundamental domain for this
relation.

\begin{Lemma}
$J$ has trivial kernel.
\end{Lemma}
\begin{proof}
Each $F\in L^2(X\times\Z_2^\Z,\mu\ot\la_{\Z_2^{\Z}})$ can be
written as
\[F(x,\underline{i})=\sum_{A\in Fin}f_A(x)(-1)^{A(\underline{i})},\mbox{ where }A(\underline{i})=\sum_{s\in A}i_s.\]
Note that $\sum_{A\in Fin}
\|f_A\|_{L^2(X,\mu)}^2=\|F\|^2_{L^2(X\times
\Z_2^{\Z},\mu\ot\la_{\Z_2^{\Z}})}$. Since \[U_{I_n}\left(f_A\ot
(-1)^{A(\cdot)}\right)
(x,\underline{i})=\left(f_A\ot(-1)^{A(\cdot)}\right)(I_n(x,\underline{i}))=f_A(S^nx)(-1)^{(\widehat{A}+n)(\underline{i})},\]
we have
\[JF(x,\underline{i})=\sum_{n\in\Z}\sum_{A\in Fin}a_n f_A(S^nx)(-1)^{(\widehat{A}+n)(\underline{i})}.\]
Notice that $n+1\notin \widehat{A}+n$. To reverse the roles played
by $A$ and $\widehat{A}+n$ note that if $B\in Fin$ and $n+1\notin
B$ then the set $\widetilde{B-n}$ is the unique set such that
$\widehat{\widetilde{B-n}}+n=B$. It follows that
\[JF(x,\underline{i})=\sum_{B\in Fin}\sum_{n\in\Z,n+1\notin B}a_nf_{\widetilde{B-n}}(S^nx)(-1)^{B(\underline{i})}=
\sum_{B\in Fin}\widetilde{F}_B(x)(-1)^{B(\underline{i})},\] where
$\widetilde{F}_B(x)=\sum_{n\in\Z,n+1\notin
B}a_nf_{\widetilde{B-n}}(S^nx)$.  For every $B\in Fin_0$ and $x\in
X$ we  define $\xi^B(x)=(\xi^B_n(x))_{n\in\Z}$ by setting
$$\xi_{-n}^B(x)=\left\{\begin{array}{ccl}f_{\widetilde{B-n}}(S^nx)&\text{ if }&n+1\notin
B\\
0&\text{ if }&n+1\in B.\end{array}\right.$$ Therefore, for
$k\in\Z$
\begin{eqnarray*}\widetilde{F}_{B+k}(x)&=&\sum_{n\in\Z,n+1\notin
B+k}a_nf_{\widetilde{B-n+k}}(S^nx)\\&=&\sum_{n\in\Z,(n-k)+1\notin
B}a_nf_{\widetilde{B-(n-k)}}(S^{-(k-n)}(S^kx))\\&=&\sum_{n\in\Z}a_n\xi^B_{k-n}(S^kx)=[\bar{a}
\ast\left(\xi^B(S^kx)\right)]_k.\end{eqnarray*}

Suppose that $J(F)=0$. It follows that given $k\in\Z$ and $B\in
Fin_0$ we have
$[\bar{a}\ast\left(\xi^B(S^kx)\right)]_k=\widetilde{F}_{B+k}(x)=0$
for $\mu$-a.e.\ $x\in X$, whence a.s.\ we also have
$[\bar{a}\ast\left(\xi^B(x)\right)]_k=0$. Letting $k$ run through
$\Z$ we obtain that  $\bar{a}\ast\left(\xi^B(x)\right)=\bar{0}$
for $\mu$-a.e.\ $x\in X$. On the other hand $\xi^B(x)\in l^2(\Z)$
for almost every $x\in X$. In view of (\ref{zalpod}),
$\xi^B(x)=\bar{0}$ for every $B\in Fin_0$ and for a.e.\ $x\in X$,
hence $f_{\widetilde{A}}=0$ for every $A\in Fin$ with $1\notin A$.
It follows that $f_{{A}}=0$ for every $A\in Fin$, consequently
$F=0$.
\end{proof}
\begin{Lemma}
$J^*$ has  trivial kernel.
\end{Lemma}
\begin{proof}
Let
\[F(x,\underline{i})=\sum_{A\in Fin}f_A(x)(-1)^{A(\underline{i})}.\]
Then
\[U^*_{I_n}\left(f_A\ot (-1)^{A(\cdot)}\right)(x,\underline{i})=
\left\{\begin{array}{ccc}f_A(S^{-n}x)(-1)^{\widetilde{A-n}(\underline{i})}&\mbox{ if }&n+1\notin A\\
0&\mbox{ if }&n+1\in A. \end{array}\right. \] It follows that
\begin{eqnarray*}J^*F(x,\underline{i})&=&\sum_{A\in Fin}\sum_{n\in\Z,n+1\notin
A}a_nf_A(S^{-n}x)(-1)^{\widetilde{A-n}(\underline{i})}\\
&=&\sum_{B\in
Fin}\sum_{n\in\Z}a_nf_{\widehat{B}+n}(S^{-n}x)(-1)^{B(\underline{i})}\\
&=&\sum_{A\in Fin,1\notin A
}\sum_{n\in\Z}a_nf_{A+n}(S^{-n}x)(-1)^{\widetilde{A}(\underline{i})}.
\end{eqnarray*}
Furthermore,
\begin{eqnarray*}
J^*F(x,\underline{i})&=&\sum_{A\in Fin_0}\sum_{k\in\Z,1\notin
A-k}\sum_{n\in\Z}a_nf_{A+n-k}(S^{-n}x)(-1)^{\widetilde{A-k}(\underline{i})}\\
&=&\sum_{A\in Fin_0}\sum_{k\in\Z,1\notin
A-k}[\bar{a}\ast\left(\zeta^A(S^{-k}x)\right)]_k(-1)^{\widetilde{A-k}(\underline{i})},\end{eqnarray*}
where $\zeta^A(x)=(\zeta^A(x)_l)_{l\in\Z}$ is given by
$\zeta^A(x)_l=f_{A-l}(S^{l}x)$.

Suppose that $J^*(F)=0$. It follows that
$[\bar{a}\ast\zeta^A(S^{-k}x)]_k=0$ for every $A\in Fin_0$,
$k+1\notin A$ and for a.e.\ $x\in X$. Hence
$\bar{a}\ast\left(\zeta^A(x)\right)\in l_0(\Z)$ for $\mu$-a.e.\
$x\in X$ (the only possibly non-zero terms of the convolved
sequence have indices belonging to $A-1$). Since $\zeta^A(x)\in
l^2(\Z)$, in view of (\ref{zalpod}), $\zeta^A(x)=\ov{0}$ for every
$A\in Fin_0$ and for $\mu$-a.e.\ $x\in X$. Thus $f_A=0$ for all
$A\in Fin$ and consequently $F=0$.
\end{proof}

It follows from the above two lemmas that the ranges of $J$ and
$J^\ast$ are dense. Clearly $J$ and $J^\ast$  intertwine the
Koopman operators $U_{T_1}$ and $U_{T_2}$, hence we have proved
the following.

\begin{Prop} Under the above notation the automorphisms $T_1$ and
$T_2$ are Markov quasi-similar.\bez
\end{Prop}

Recall that in \cite{Le}  constructions of the above type have
been used to produce weakly isomorphic transformations that are
not isomorphic. In fact our transformation $T_1$ is the same as
the transformation $T_{\ldots,-1,0,1,2,\ldots}$ in Subsection~4.2
in \cite{Le}, where it is proved that each metric endomorphism
that commutes with $T_1$ is invertible. It follows that $T_1$
cannot be a factor of the system given by its {\bf proper} factor;
in particular, it is not weakly isomorphic to $T_2$. In other
words we have proved the following.

\begin{Prop}\label{vershik-answer}
There are ergodic automorphisms which are Markov
quasi-similar  but they are not weakly isomorphic.\bez\end{Prop}

\begin{Remark}The Markov quasi-similarity between $T_1$ and $T_2$ constructed above
is given by a $1-1$ Markov operator with dense range, that is, in
fact we have shown that $U_{T_1}$ and $U_{T_2}$ are Markov
quasi-affine. The Markov operator is given as a convex combination
of isometries which separately have no dense ranges as they are
not onto (and obviously their ranges are closed). Let us emphasize
that not each non-trivial choice of weights $(a_n)$ gives rise to
an operator with dense range as the following example shows.
\end{Remark}

\begin{Example}
Set $a_n=\frac1{2^{n+1}}$ for $n\geq 0$ and $a_n=0$ for $n<0$. We
will show that in this case $ker\, J^\ast\neq\{0\}$. Denoting by
$\overline{S}$ the automorphism of
$(X\times\Z_2^\Z,\mu\ot\la_{\Z_2^{\Z}})$ given by
\[\overline{S}(x,\underline{i})=(Sx,\ldots,i_{-1},
i_0,\stackrel{0}{i_1},i_2,\ldots),\] we have $I_n=I_0\circ
\overline{S}^n$ for any $n\in\Z$, and hence
\[J^*=U^{\ast}_{I_0}\circ\sum_{n=0}^{\infty}\frac{1}{2^{n+1}}U_{\overline{S}^{-n}}.\] In fact, we will prove that
\begin{equation}\label{o7}
\left(-\frac12U_{\overline{S}^{-1}}+Id\right)\left(ker\,U^\ast_{I_0}\right)\subset
ker\,J^\ast.
\end{equation}
Notice that if $0\neq G\in
L^2(X\times\Z_2^\Z,\mu\ot\la_{\Z_2}^\Z)$ then $-\frac12G\circ
\ov{S}^{-1}+G\neq 0$ because the norms of the two summands are
different. To prove (\ref{o7}) take $G\in ker\, U_{I_0}^\ast$ and
let $F=-\frac12 G\circ \ov{S}^{-1}+G$. Thus
\begin{eqnarray*}
J^*F&=&U^{\ast}_{I_0}\left(\sum_{n=0}^{\infty}\frac{1}{2^{n+1}}F\circ
\overline{S}^{-n}\right)\\&=&U^{\ast}_{I_0}\sum_{n=0}^{\infty}\left(\frac{1}{2^{n+1}}G\circ
\overline{S}^{-n}-\frac{1}{2^{n+2}}G\circ
\overline{S}^{-n-1}\right)
=U^{\ast}_{I_0}\left(\frac{1}{2}G\right)=0.
\end{eqnarray*}
Since $ker\,U^\ast_{I_0}$ is not trivial, the claim follows.
\end{Example}

\section{Metric invariants of Markov quasi-similarity}\label{AA}
By Proposition~\ref{equivalence} the Markov quasi-similarity is
stronger than spectral equivalence of Koopman representations (it
will be clear from the results of this section that it is
essentially stronger). In particular all spectral invariants like
ergodicity, weak mixing, mild mixing, mixing and rigidity are
invariants for Markov quasi-similarity. It also follows that each
transformation which is spectrally determined, that is for which
spectral equivalence is the same as measure-theoretical
equivalence, is also Markov quasi-equivalence unique (up to
measure-theoretic isomorphism). In particular each automorphism
Markov quasi-similar to an ergodic transformation with discrete
spectrum is isomorphic to it. The same holds for
Gaussian-Kronecker systems (see \cite{Fo-St}).

This spectral flavor is still persistent when we consider Markov
quasi-images. Indeed,  each Markov operator between $L^2$-spaces
``preserves'' the subspace of zero mean functions, therefore a
direct consequence of~(\ref{q3}) is that a transformation which is
a Markov quasi-image of an ergodic (weakly mixing, mixing) system
remains ergodic (weakly mixing, mixing). Despite all this, Markov
quasi-similarity is far from being spectral equivalence. In order
to justify this statement, we need a non-disjointness result from
\cite{Le-Pa-Th} (in fact its proof) which we now briefly recall.

Assume that $T_i$ is an ergodic automorphism of
$(X_i,\cb_i,\mu_i)$, $i=1,2$ and let $\Phi:L^2(X_1,\cb_1,\mu_1)\to
L^2(X_2,\cb_2,\mu_2)$ be a Markov operator intertwining $U_{T_1}$
and $U_{T_2}$. Then $\Phi$ sends $L^\infty$-functions to
$L^\infty$-functions and we can consider $H_\Phi$, the $L^2$-span
of
$$
\{\Phi(f_1^{(1)})\cdot\ldots\cdot \Phi(f^{(1)}_m):\: f^{(1)}_i\in
L^\infty(X_1,\cb_1,\mu_1),\;i=1,\ldots,m,\;m\geq1\}.$$ It turns
out that $H_\Phi=L^2(\ca_\Phi)$ where $\ca_\Phi\subset\cb_2$ is a
$T_2$-invariant $\sigma$-algebra (in other words $\Phi$ defines a
factor of $T_2$). Then by the proof of the main non-disjointness
result (Theorem~4) in \cite{Le-Pa-Th}  this factor is also a
factor of an (ergodic) infinite self-joining of $T_1$. If we
assume additionally that Im$\,\Phi$ is dense then
$H_\Phi=L^2(X_2,\cb_2,\mu_2)$ and the factor given by $\ca_\Phi$
is equal to $T_2$ itself.
\begin{Prop}\label{niezm}
If  $T_2$ is a Markov quasi-image of $T_1$ then $T_2$ is a factor
of some infinite ergodic self-joinings of $T_1$.\bez
\end{Prop}

As all the systems determined by (infinite) joinings of zero
entropy systems have zero entropy and the systems given by
joinings of distal systems are also distal (for these results see
e.g.\ \cite{Gl}), Proposition~\ref{niezm} yields the following
conclusion.

\begin{Prop} Each automorphism which is a Markov quasi-image
of a zero entropy system has zero entropy.  Each automorphism
which is a Markov quasi-image of a distal system remains distal.
In particular, zero entropy and distality are invariants of Markov
quasi-similarity in the class of measure-preserving systems.\bez
\end{Prop}

As a matter of fact, we can prove that zero entropy is an
invariant of Markov quasi-similarity in the class of
measure-preserving systems in a simpler manner. Recall that $T_1$
and $T_2$ are said to be {\em disjoint} (in the sense of
Furstenberg \cite{Fu}) if the only joining between them is the
product measure. The following result will help us to indicate
further invariants of Markov quasi-similarity.
\begin{Lemma}\label{disj}
If $T_1$ is disjoint from $S$ and $T_2$ is a Markov quasi-image of
$T_1$ then $T_2$ is also disjoint from $S$.
\end{Lemma}

\begin{proof}Indeed, assume that $\Phi\circ U_{T_1}=U_{T_2}\circ \Phi$ and
$\Phi$ has dense range. If $T_2$ and $S$ are not disjoint then we
have a non-trivial Markov operator $\Psi$ intertwining $U_{T_2}$
and $U_S$. Since $\Phi$ has  dense range, $\Psi\circ\Phi$ is a
non-trivial Markov operator intertwining $U_{T_1}$ and $U_S$ and
therefore $T_1$ is not disjoint from $S$.
\end{proof}
Given a class $\cm$ of automorphisms denote by $\cm^\perp$ the
class of those transformations which are disjoint from all members
of $\cm$. In view of Lemma~\ref{disj}  we have the following.
\begin{Prop}\label{perp}
$\cm^\perp$ is closed under taking automorphisms which are Markov
quasi-images of members of $\cm^\perp$. In particular, if
$\cm=\cm^{\perp\perp}$ then $\cm$ is closed under taking
automorphisms which are Markov quasi-images of members of $\cm$.
\bez
\end{Prop}

If by $\mathcal{K}$ and $\mathcal{ZE}$ we denote the classes of
Kolmogorov automorphisms and zero entropy automorphisms
respectively then we have $\mathcal{K}=\mathcal{ZE}^\perp$
(\cite{Fu}) and therefore by Proposition~\ref{perp}
 we obtain the following.
\begin{Cor}
Every automorphism which is a Markov quasi-image  of a Kolmogorov
automorphism is also~K. In particular, $K$ property is an
invariant of Markov quasi-similarity in the class of
measure-preserving systems.\bez
\end{Cor}

\begin{Prob} Is the same true for Bernoulli automorphisms?
\end{Prob}

Notice that also $\mathcal{ZE}=\mathcal{K}^\perp$. Therefore we
can apply Proposition~\ref{perp} with $\mathcal{M}=\mathcal{ZE}$
to obtain that an automorphism which is a Markov quasi-image of a
zero entropy system has zero entropy.

\section{JP property and
Markov quasi-similarity}
\begin{Def}An ergodic automorphism $T$ on
$\xbm$ is said to have {\em the joining primeness} (JP) property
(see \cite{Le-Pa-Ro}) if for each pair of weakly mixing
automorphisms $S_1$ on $(Y_1,\mathcal{C}_1,\nu_1)$ and $S_2$ on
$(Y_2,\mathcal{C}_2,\nu_2)$ and for every indecomposable Markov
operator
$$\Phi: L^2(X,\mu)\to L^2(Y_1\times Y_2,\nu_1\otimes\nu_2)$$
intertwining $U_T$ and $U_{S_1\times S_2}$ we have (up to some
abuse of notation)  Im$\,\Phi\subset L^2(Y_1,\cc_1,\nu_1)$ or
Im$\,\Phi\subset L^2(Y_2,\cc_2,\nu_2)$.
\end{Def}

The class of JP automorphisms includes in particular the class of
simple systems (\cite{Ju-Ru}). For other natural classes of JP
automorphisms including some smooth systems see \cite{Le-Pa-Ro}
(we should however emphasize that a ``typical'' automorphism is JP
\cite{Le-Pa-Ro}).

Assume that $T$ is JP and $S_1, S_2,\ldots $ are weakly mixing.
Let $\Phi: L^2(X,\mu)\to L^2(Y_1\times
Y_2\times\ldots,\nu_1\otimes\nu_2\otimes\ldots)$ be a Markov
operator intertwining $U_T$ and $U_{S_1\times S_2\times\ldots}$.
Let $\Phi=\int_\Gamma\Phi_\gamma\,dP(\gamma)$ be the decomposition
corresponding to the ergodic decomposition of the joining
determined by $\Phi$. Slightly abusing notation, we claim that for
$P$-a.e.\ $\gamma\in\Gamma$
$$
\Phi_\gamma(L^2\xbm)\subset
L^2(Y_{i_\gamma},\cc_{i_\gamma},\nu_{i_\gamma}),\text{ for some }
i_\gamma\in\{1,2,\ldots\}.$$ Indeed, we use repeatedly the
definition of JP property: We represent $\Pi_{n\geq1}S_n$ as
$S_1\times\left(\Pi_{n\geq2}S_n\right)$ and if Im$\,\Phi_\gamma$
is not included in $L^2(Y_1,\nu_1)$ then Im$\,\Phi_\gamma\subset
L^2(Y_2\times Y_3\times\ldots,\nu_2\otimes\nu_3\ot\ldots)$. In the
next step we write $\Pi_{n\geq1}S_n=\left(S_1\times
S_2\right)\times\left(\Pi_{n\geq3}S_n\right)$ and we check if
Im$\,\Phi_\gamma\subset L^2(Y_1\times Y_2,\nu_1\ot\nu_2)$ (if it
is the case then Im$\,\Phi_\gamma\subset L^2(Y_2,\nu_2)$); if it
is not the case then Im$\,\Phi_\gamma\subset L^2(Y_3\times
Y_4\times\ldots,\nu_3\ot\nu_4\ot\ldots)$, etc. If for each
$n\geq1$, Im$\,\Phi_\gamma\perp L^2(Y_1\times\ldots\times
Y_n,\nu_1\ot\ldots\ot\nu_n)$, then  Im$\,\Phi_\gamma=0$ (since
functions depending on finitely many coordinates are dense), and
hence $\Phi_\gamma=0$.

It follows that for some $0\leq a_n\leq1$ with
$\sum_{n\geq1}a_n=1$
\begin{equation}\label{opisJPP}
\Phi=\sum_{n\geq1}a_n\Phi_n,\end{equation} where
Im$\,\Phi_n\subset L^2(Y_n,\cc_n,\nu_n)$. In particular,
\begin{equation}\label{jeszcze1cud}
\mbox{Im}\,\Phi\subset \bigoplus_{n\geq1} L^2(Y_n,\cc_n,\nu_n)
\subset L^2(Y_1\times
Y_2\times\ldots,\cc_1\ot\cc_2\ot\ldots,\nu_1\ot\nu_2\ot\ldots).\end{equation}
Note that the space $F:=\bigoplus L^2(Y_n,\nu_n)$ is closed and
$U_{S_1\times S_2\times\ldots}$-invariant.

\begin{Lemma}\label{opisJPP1} Under the above notation,
if  $\ca\subset\cc_1\ot\cc_2\ot\ldots$ is a factor
 of $S_1\times S_2\times\ldots$ and it is also
a Markov quasi-image of a JP automorphism $T$ then there exists
$n_0\geq1$ such that $\ca\subset\cc_{n_0}$; in other words the
factor given by $\ca$ is a factor of $S_{n_0}$.
\end{Lemma}
\begin{proof} Asume that
$\Phi$ intertwines $U_T$ and the Koopman operator of the factor
action of $S_1\times S_2\times\ldots$ on $\ca$. Since the range of
$\Phi$ is dense in $L^2(\ca)$, it follows that $\Phi:L^2\xbm\to
L^2(\ca)\subset F$. We now use an argument from \cite{Ju-Le}: Take
$A\in\ca$. In view of~(\ref{jeszcze1cud}) we have
$$
{\mathbf
1}_A-(\nu_1\ot\nu_2\ot\ldots)(A)=f_1(y_1)+f_2(y_2)+\ldots$$ with
$f_n\in L^2_0(Y_n,\nu_n)$, $n\geq1$. Since the distribution of the
random variable ${\mathbf 1}_A-(\nu_1\ot\nu_2\ot\ldots)(A)$ is a
measure on a two element set and the random variables
$f_1,f_2,\ldots$ are independent, all of them but one, say
$f_{n_A}$, are equal to zero. In other words, $A\in\cc_{n_A}$. It
easily follows that the function $\ca\ni A\mapsto n_A$ is constant
(see \cite{Ju-Le}).
\end{proof}

Let $T$ be a simple weakly mixing automorphism. By the definition
of simplicity, it follows that each of its ergodic infinite
self-joinings is, as a dynamical system, isomorphic to a Cartesian
product $T^{\times n}$ with $n\leq\infty$. Since each simple
system has the JP property, in view of~Proposition~\ref{niezm} and
Lemma~\ref{opisJPP1} (in which $S_n=T$) we obtain the following.

\begin{Prop}\label{simple-Markov} Each automorphism which is a Markov quasi-image of a
simple map $T$ is a factor of $T$.\bez\end{Prop}

It follows from the above proposition that if $T_1$ and $T_2$ are
weakly mixing simple automorphisms and  are Markov quasi-similar
then they are isomorphic.


\begin{Remark} In our example of $T_1$ and $T_2$ non-weakly isomorphic but
Markov-quasi-similar $T_2$ is a factor of $T_1$ but (because of
absence of weak isomorphism) $T_1$ is not a factor of $T_2$. Hence
the family of factors of $T_2$ is strictly included in the family
of automorphisms which are Markov quasi-images of $T_2$.
\end{Remark}

When we apply Proposition~\ref{simple-Markov} to the MSJ maps (see
\cite{Ju-Ru}) we obtain that such systems are Markov
quasi-similarly prime, that is we have the following.

\begin{Cor}\label{MSJMarkov} The only non-trivial automorphism which
is a Markov quasi-image of an MSJ system $T$ is $T$ itself.\bez
\end{Cor}

\begin{Remark}
Assume that $T$ enjoys the MSJ property. Take $\Phi_1,\Phi_2$ two
joinings of $T$ and $T\times T$ so that Im$\,\Phi_1\cap
\left(L^2(X,\mu)\ot{\mathbf 1}_X\right)\neq\{0\}$ and
Im$\,\Phi_2\cap \left({\mathbf 1}_X\ot
L^2(X,\mu)\right)\neq\{0\}$. Then  $\Phi:=a\Phi_1+(1-a)\Phi_2$ is
a Markov operator intertwining  $U_T$ and $U_{T\times T}$ and if
$0<a<1$, then the range of $\Phi$ is {\bf not} dense in
$L^2(\ca_\Phi)$. Indeed, $\ca_\Phi$ is either $T\times T$ or
$T\odot T$ (the factor of $T\times T$ determined by the
$\sigma$-algebra of sets invariant under exchange of coordinates)
and the claim follows from Lemma~\ref{opisJPP1}. This is the
answer to a question raised by Fran\c{c}ois Parreau in a
conversation with the second named author of the note.

It means that if we try to define Markov quasi-image by requiring
that $\ca_\Phi=\cb_2$ instead of requiring that the range of
$\Phi$ is dense in $L^2(X_2,\cb_2,\mu_2)$ then we obtain a
strictly weaker notion.
\end{Remark}

\section{Final remarks and problems}\label{BB}
 Notice that the joining of $T_1$ and $T_2$ corresponding to the
Markov operator in Section~\ref{przyklad} and based on
constructions from \cite{Le} is not ergodic (i.e.\ the Markov
operator is decomposable). In fact,  in our construction of two
non-weakly isomorphic Markov quasi-similar automorphisms $T_1$ and
$T_2$ no Markov operator corresponding to an ergodic joining
between $T_1$ and $T_2$ can have dense range. Indeed, first recall
that ergodic Markov quasi-similar automorphisms have the same
Kronecker factors. Then notice that $T_1$ and $T_2$ are compact
abelian group extensions of the same (in \cite{Le} this is the
classical adding machine system) Kronecker factor. Hence, assume
that $T$ is an ergodic automorphism with discrete spectrum and let
$\phi:X\to G$, $\psi:X\to H$ be ergodic cocycles with values in
compact abelian groups $G$ and $H$ respectively. We then have the
following.
\begin{equation}\label{groupext}
\begin{aligned} \mbox{$T_\phi$ and $T_\psi$ are Markov quasi-similar via {\bf
indecomposable}}\\
\mbox{Markov operators if and only if they are weakly isomorphic.}
\end{aligned}\end{equation}
Indeed, every ergodic joining between such systems is the
relatively independent extension of the graph joining given by an
isomorphism $I$ of so called natural factors $T_{\phi J}$ and
$T_{\psi F}$ acting on $X\times G/J$ and $X\times H/F$
respectively, see \cite{Le-Me}. The Markov operator $\Phi$
corresponding to such a joining is determined by the orthogonal
projection on the $L^2(X\times H/F,\mu\ot\la_{H/F})$; in
particular the range of $\Phi$ is closed. Therefore it has  dense
range only if  Im$\,\Phi=L^2(X\times H,\mu\ot\la_H)$ which means
that in fact $I$ settles a metric isomorphism of $T_\psi$ and a
factor of $T_\phi$. In other words, $T_\psi$ is a factor of
$T_\phi$.

This shows that there exist two ergodic automorphisms which are
Markov quasi-similar but  Markov quasi-similarity cannot be
realized by indecomposable Markov operators with dense ranges.

We have been unable to construct an indecomposable 1-1 Markov
operator $\Phi$ with dense range intertwining the Koopman
operatros given by two non-isomorphic ergodic automorphisms $T_1$
and $T_2$. One might think about such a construction using Markov
operators given as convex combinations of $U_{S_i}$ where $S_i$
are space isomorphisms which are {\bf not} intertwining $T_1$ and
$T_2$ (see e.g.\ \cite{Da} for the notion of near simplicity where
similar idea is applied).

It seems that Proposition~\ref{equivalence} rules out a
possibility to find two  Markov quasi-similar Gaussian
automorphisms which are not isomorphic by a use of so called
Gaussian joinings \cite{Le-Pa-Th} (recall that Gaussian joinings
are ergodic joinings). Indeed, once a Markov quasi-similarity is
given by an integral of Markov operators corresponding to Gaussian
joinings, it sends chaos into chaos (see \cite{Le-Pa-Th} for
details). In particular, first chaos is sent into first chaos, and
we obtain quasi-similarity of the unitary actions restricted to
the first chaos. By Proposition~\ref{equivalence} these actions on
the first chaos are spectrally equivalent which in turn implies
measure-theoretic isomorphism of the Gaussian systems.

We do not know however if we can have two non-weakly isomorphic
Poisson suspension systems  which are Markov quasi-similar by a
use of Poissonian joinings (which are ergodic), see
\cite{De-Fr-Le-Pa} and \cite{Ro}.

\begin{Prob} Recall  that in the construction carried out in
Section~\ref{przyklad}, $T_2$ was a factor of $T_1$. Is it
possible to construct Markov quasi-similar automorphisms $T_1$ and
$T_2$ such that $T_1$ and $T_2$ have no common (non-trivial)
factors? Of course such $T_1$ and $T_2$ must not be disjoint (see
\cite{Fu}).

The most ``popular'' construction of a pair of non-disjoint
systems without common factors is $(T, T\odot T)$ (for a
particular $T$; see \cite{Ju-Le}, \cite{Ru}). Notice however that
these two automorphisms are not Markov quasi-similar if $T$ has
the JP property (see Lemma~\ref{opisJPP1}), that is, in all known
cases where $T$ and $T\odot T$  have no common (isomorphic)
non-trivial factors.
\end{Prob}

\begin{Prob} As we have already noticed in Remark~\ref{similarity}, Markov
quasi-affinity implies Markov quasi-similarity. Are these notions
equivalent? If the answer is positive then each weakly isomorphic
transformations would have to be Markov quasi-affine. Are examples
of weakly isomorphic non-isomorphic automorphisms from
\cite{Kw-Le-Ru}, \cite{Le} or \cite{Ru} Markov quasi-affine?
\end{Prob}

\begin{Prob} The examples of Markov quasi-similar automorphisms which are
not isomorphic presented in this note have infinite spectral
multiplicity. Is it possible to find such examples in the class of
systems with simple spectrum (or of finite spectral multiplicity)?
In the class of rank one systems? Recall that in case of finite
spectral multiplicity systems their weak isomorphism implies
isomorphism, see e.g.\ \cite{Ne}.
\end{Prob}
\section*{Acknowledgements} The authors would like to thank Vitaly
Bergelson for fruitful discussions and stimulating questions on
the subject.  We would like also to thank Alexander Gomilko for
his remarks on the content of Section~\ref{calkowanie}.

\end{document}